\documentclass[oneside]{amsart}
\usepackage{amsmath, amssymb}
\usepackage{amsthm}
\newtheorem{theorem}{Theorem}[section]
\newtheorem{proposition}[theorem]{Proposition}
\newtheorem{lemma}[theorem]{Lemma}

\theoremstyle{definition}

\def\F{\mathcal{F} }

\def\nbd{neighborhood }

\begin{document}

\title[$R$-closedness and Upper semicontinuity]
{$R$-closedness and Upper semicontinuity}
\author{Tomoo Yokoyama}
\date{\today}
\email{yokoyama@math.sci.hokudai.ac.jp}

\thanks{The author is partially supported 
by the JST CREST Program at Department of Mathematics,  
Hokkaido University.}

\maketitle

\begin{abstract} 
Let $\mathcal{F} $ be a pointwise almost periodic decomposition 
of a compact metrizable space $X$.  
Then we show that 
$\mathcal{F} $ is $R$-closed  
if and only if 
$\hat{\mathcal{F}} $ is usc.  
On the other hand, 
let $G$ be a flow on a compact metrizable space 
and $H$ a finite index normal subgroup.
Then 
$G$ is $R$-closed  
if and only if so is $H$.   
Moreover,  
if there is a finite index normal subgroup $H$ of 
an $R$-closed flow $G$ on a compact manifold 
such that the orbit closures of $H$ consist of 
codimension $k$ compact connected submanifolds and  
``few singularities'' for $k = 1$ or $2$,  
then the orbit class space of $G$ 
is a compact $k$-dimensional manifold with conners.  
In addition, 
let $v$ be a nontrivial $R$-closed vector field on a connected compact $3$-manifold $M$. 
Then one of the following holds: 
1)
The orbit class space $M/ \hat{v}$ is $[0,1]$ or $S^1$ and 
each interior point of $M/ \hat{v}$  is two dimensional. 
2) 
$\mathrm{Per}(v)$ is open dense 
and 
$M = \mathrm{Sing}(v) \sqcup \mathrm{Per}(v)$. 
3)
There is a nontrivial non-toral minimal set. 
\end{abstract}

\section{Preliminaries}
In \cite{ES}, 
they have show that 
if a continuous mapping $f$ of a topological space $X$ 
in itself 
is either pointwise recurrent or 
pointwise almost periodic   
then so is $f^k$ for each positive integer $k$. 
This results is extended into flow cases  
(see Theorem 2.24, 4.04, and 7.04 \cite{GH}). 
In \cite{Y3}, 
one have shown 
the analogous result for $R$-closed homeomorphisms.  
In this paper, 
we extend into the $R$-closed flow cases. 

The leaf space of a compact continuous codimension two foliation $\mathcal{F} $ 
of a compact manifold $M$ 
is a compact orbifold 
\cite{E}, 
\cite{EMS}, 
\cite{E2}, 
\cite{V}, 
\cite{V2}.  
%
On the other hand, there are 
non-$R$-closed compact foliations and 
non $R$-closed flows each of whose orbits is compact 
for codimension $q > 2$ 
\cite{S}, 
\cite{EV},  
\cite{V3}. 
In \cite{Y2}, 
the author has shown that 
the set of $R$-closed decompositions on compact manifolds 
contains properly the 
set of codimension one or two foliations which is minimal or compact. 
In this paper, 
for $k = 1$ or $2$,  
we show that 
the quotient space of a codimension $k$ compact connected decomposition  
with ``few singularities'' defined by a flow   
is a compact $k$-dimensional manifold.  
In addition, 
let $v$ a nontrivial $R$-closed vector field on a connected compact $3$-manifold. 
Then one of the following holds: 
1)
The orbit class space $M/ \hat{v}$ is $[0,1]$ or $S^1$ and 
each interior point of $M/ \hat{v}$  is two dimensional. 
2) 
$\mathrm{Per}(v)$ is open dense 
and 
$M = \mathrm{Sing}(v) \sqcup \mathrm{Per}(v)$. 
3)
There is a nontrivial non-toral minimal set. 

By a decomposition, 
we mean a family $\mathcal{F} $ of pairwise disjoint nonempty subsets of a set $X$ 
such that $X = \sqcup \mathcal{F}$. 
Let $\mathcal{F} $ be a decomposition of a topological space $X$. 
For any $x \in X$, 
denote by $L_x$ or $\F(x)$ the element of $\mathcal{F} $ containing $x$. 
For a subset $A \subseteq X$, 
$A$ is saturated if 
$A = \sqcup_{x \in A} L_x$. 
$\mathcal{F} $ 
is upper semicontinuous (usc) if 
each element of $\mathcal{F} $ is both closed and compact 
and, 
for any $L \in \mathcal{F} $ and for any open neighbourhood $U$ of $L$, 
there is a saturated neighbourhood of $L$ contained in $U$. 
Note that we can choose 
the above $U$ open. 
$\mathcal{F} $ is $R$-closed if 
$R := \{ (x, y) \mid y \in \overline{L_x} \}$ is closed. 
%
%
$\mathcal{F} $ is pointwise almost periodic if 
the set of all closures of elements of $\mathcal{F} $ also is a decomposition. 
Then denote by $\hat{\mathcal{F}} $ the decomposition of closures 
and by $M/\hat{\mathcal{F}}$ the quotient space, called the orbit class space. 
By a flow, we mean a continuous action of a topological group $G$ on 
a topological space $X$. 
For a flow $G$, 
denote by $\mathcal{F} _G$ the set of orbits of $G$. 
Recall 
a flow $G$ is $R$-closed if 
the set $\mathcal{F} _G$ of orbits is an $R$-closed decomposition. 
Then $G$ is $R$-closed if and only if 
 $R := \{ (x, y) \mid y \in \overline{G(x)} \}$ is closed. 
Recall that 
each $R$-closed decomposition is pointwise almost periodic. 
For an $R$-closed vector field $v$, 
write $M/\hat{v} :=M/\hat{\mathcal{F}}_{v}$.  

\section{$R$-closedness and Upper semi continuity}

Now we show the following key lemma. 

\begin{lemma}\label{lem:001}
Let $\mathcal{F} $ be a decomposition of a Hausdorff space $X$.  
If 
$\mathcal{F} $ is pointwise almost periodic
and 
$\hat{\mathcal{F}} $ is usc, 
then $\mathcal{F} $ is $R$-closed. 
If $X$ is compact metrizable and $\mathcal{F} $ is $R$-closed, 
then 
$\mathcal{F} $ is pointwise almost periodic 
and 
$\hat{\mathcal{F}} $ is usc. 
\end{lemma}

\begin{proof}
Suppose that 
$\mathcal{F} $ is pointwise almost periodic 
and 
that 
$\hat{\mathcal{F}} $ is usc.  
By Proposition 1.2.1 \cite{D} (p.13), 
we have that 
$X/\mathcal{F} $ is Hausdorff. 
By Lemma 2.3 \cite{Y}, 
we obtain that 
$\mathcal{F} $ is $R$-closed. 
Conversely, 
suppose that 
$X$ is compact metrizable and 
$\mathcal{F} $ is $R$-closed. 
By Lemma 1.6 \cite{Y}, 
we have that 
$\mathcal{F} $ is pointwise almost periodic and 
that 
the quotient map $p: X \to X/\hat{\mathcal{F}} $ is closed. 
Since $X$ is compact Hausdorff, 
we obtain that 
each element of $\hat{\mathcal{F}} $ is compact. 
By Proposition 1.1.1 \cite{D}(p.8), 
we have that 
$\mathcal{F} $ is usc.
 \end{proof}

Lemma 2.3 \cite{Y} implies the following result. 
 
\begin{proposition}\label{prop:002}
Let $\mathcal{F} $ be a pointwise almost periodic decomposition 
of a compact metrizable space $X$.  
The following are equivalent: 
\\
1) 
$\mathcal{F} $ is $R$-closed.  
\\
2) 
$\hat{\mathcal{F}} $ is usc.  
\\
3) 
$X/\hat{\mathcal{F}} $ is Hausdorff.  
\end{proposition}

\section{Inherited properties of $R$-closed flows}

Recall that 
a subgroup $H$ of a topological group $G$ is 
syndetic if there is a compact subset of $G$ 
such that $K \cdot H = G$. 

\begin{lemma}\label{lem31}
Let $G$ be a flow on a topological space $X$ and 
$H$ a syndetic subgroup of $G$. 
If $H$ is $R$-closed, 
then so is $G$. 
Moreover  
$\overline{G \cdot x} = K \cdot \overline{H \cdot x}$ 
for any $x \in Y$ where 
$K$ is a compact subgroup with $K \cdot H = G$. 
\end{lemma}

\begin{proof} 
For any flow $\pi: G \times Y \to Y$ on a topological space $Y$, 
we claim that 
$K \cdot C$ is closed for a closed subset $C$ of $Y$. 
Indeed, 
fix a point $x \in Y - K \cdot C$. 
Then $Y - C$ is an open \nbd of $K^{-1} \cdot x$ and so 
$\pi^{-1}(Y - C)$ is an open \nbd of $K^{-1} \times \{ x \}$. 
Since $K^{-1} \times \{ x \}$ is compact, 
by the tube theorem, 
there is an open \nbd $U$ of $x$ 
such that 
$K^{-1} \times U \subseteq \pi^{-1}(Y - C)$. 
Then 
$K^{-1} \cdot U \subseteq Y - C$. 
Therefore 
$(K^{-1} \cdot U) \cap C = \emptyset$ 
and so 
$U \cap (K \cdot C) = \emptyset$. 
This shows that $K \cdot C$ is closed. 
Let 
$R_G := \{ (x, y) \mid y \in \overline{G \cdot x} \}$ 
and 
$R_H := \{ (x, y) \mid y \in \overline{H \cdot x} \}$. 
Suppose that 
$H$ is $R$-closed.  
Then 
$ 
{K \cdot \overline{H \cdot x}} \subseteq 
\overline{G \cdot x} = 
\overline{K \cdot (H \cdot x)} \subseteq  
\overline{K \cdot \overline{H \cdot x}}$. 
By the above claim, 
we have 
$
\overline{G \cdot x} =
\overline{(K \cdot H) \cdot x} = 
\overline{K \cdot \overline{H \cdot x}} = 
{K \cdot \overline{H \cdot x}}$.  
Consider an action of $G$ on $X \times X$ by $g \cdot (x, y) := (x , g^{-1} \cdot y)$.  
Since 
$K$ is compact and $R_H$ is closed,  
the above claim implies that 
$R_G = K \cdot R_H$ is closed.   
\end{proof}

For a subset $V$, write $\hat{\F}(V) := \mathrm{Sat}_{\hat{\F}}(V) = 
\cup_{x \in V} \hat{\F}(x)$. 
We generalize Lemma 1.1 \cite{Y3} to flows. 
This statement is an analogous result for recurrence and pointwise almost periodicity 
(see Theorem 2.24, 4.04, and 7.04 \cite{GH}). 

\begin{lemma}\label{prop21}
Let $G$ a flow on a 
compact metrizable space $X$ 
and 
$H$ a  finite index normal subgroup. 
Then 
$G$ is $R$-closed  
if and only if so is $H$. 
\end{lemma}

\begin{proof} 
By Lemma \ref{lem31}, 
the $R$-closedness of $H$ implies one of $G$. 
Conversely, 
suppose that $G$ is $R$-closed. 
Let $n$ be the index of $H$ and 
$\{ f_1, \dots , f_{n-1} \}$ a subset of $G$ 
such that 
$G = H \sqcup H f_1  \sqcup \cdots \sqcup H f_{n-1}$. 
Put 
$\F := \F_{G}$. 
%
By Corollary 1.4 \cite{Y}, 
we have that 
 $G$ is pointwise almost periodic. 
%
By Theorem 2.24 \cite{GH}, 
we have that 
$H$ is also pointwise almost periodic.  
By Proposition \ref{prop:002}, 
$\hat{\F}$ is usc 
and 
it suffices to show that 
$\hat{\F}_H$ is usc. 
Note that 
$\F_H (x) \subseteq \F (x) $ and so 
$\hat{\F}_H (x)  \subseteq \hat{\F} (x) $. 
For $x \in X$ with $\hat{\F}_H (x) = \hat{\F} (x)$ and 
for any open neighbourhood 
$U $ of $\hat{\F} (x) = \hat{\F}_H (x)$, 
since $\hat{\F}$ is usc, 
there is a 
$\hat{\F}$-saturated open neighbourhood  $V$ of $\hat{\F} (x)$ 
contained in $U$. 
Since $\hat{\F}_H (x)  \subseteq \hat{\F} (x) $, 
we have that 
$V$ is also 
a ${\hat{\F}_H}$-saturated open neighbourhood  $V$ of $\hat{\F} (x)$. 
Fix any $x \in X$ with $\hat{\F} (x) \neq \hat{\F}_H (x)$. 
Put 
$\hat{L}_1 := \hat{\F}_H (x)$ and 
$\{ \hat{L}_2, \dots , \hat{L}_k \} := 
\{  \hat{\F} (f_l(x)) \mid l = 1, \dots , n-1 \}$ 
such that 
$\hat{L}_i \cap \hat{L}_j = \emptyset$
 for any $i \neq j \in \{1, \dots , k\}$. 
%
Let $\hat{L}' = \hat{L}_2 \sqcup \dots \sqcup \hat{L}_k$. 
Then 
$\hat{L}_1$ and $\hat{L}'$ are closed
and 
$\hat{\F}(x) 
= \hat{L}_1 \sqcup \cdots \sqcup \hat{L}_k 
= \hat{L}_1 \sqcup \hat{L}'$. 
For any sufficiently small $\varepsilon  > 0$, 
let 
$U_{1, \varepsilon } = B_{\varepsilon } (\hat{L}_1)$ 
(resp. $U_{\varepsilon }' = B_{\varepsilon } (\hat{L}')$) 
be the open $\varepsilon $-ball of $\hat{L}_1$ 
(resp. $\hat{L}'$). 
Then 
$\overline{U_{1, \varepsilon /2}} \subseteq U_{1, \varepsilon }$  
and 
$\overline{U_{\varepsilon/2}'} \subseteq U_{\varepsilon }'$. 
Since $\varepsilon $ is small and 
$X$ is normal,  
we obtain 
$U_{1, \varepsilon } \cap U_{\varepsilon }' = \emptyset$. 
Since $\hat{\F}$ is usc, 
there are neighbourhoods $V_{1, \varepsilon } \subseteq U_{1, \varepsilon/2}$ 
(resp. $V'_{\varepsilon } \subseteq U_{\varepsilon/2}'$) of $\hat{L}_1$ 
(resp. $\hat{L}'$) 
such that 
$V_{1, \varepsilon } \sqcup V'_{\varepsilon }$ is an $\hat{\F}$-saturated neighbourhood of $\hat{\F} (x)$. 
Since $\hat{\F}(x)$ is compact 
and 
$V_{1, \varepsilon} \sqcup V'_{\varepsilon}$ is an open neighbourhood of 
$\hat{\F}(x)$, 
there are finitely many connected components 
of $V_{1, \varepsilon} \sqcup V'_{\varepsilon}$ 
each of which intersects $\hat{\F}(x)$ 
and 
whose union is also a covering of $\hat{\F}(x)$. 
Let $W_1$ be the finite union in $V_{1, \varepsilon} \sqcup V'_{\varepsilon}$. 
Since 
$W_1 \supseteq \hat{\F}(x)$ 
and 
$g(C) \cap \hat{\F}(x) \neq \emptyset$ for any $g \in G$ and 
any connected component $C$ of $W_1$, 
we have 
$G(W_{1})$ also consists of finitely many connected components. 
Since 
$G(W_1) 
\subseteq 
{\hat{\F}}(W_{1})
\subseteq 
\overline{G(W_{1})}$, 
we have that 
${\hat{\F}}(W_{1})  \subseteq V_{1, \varepsilon } \sqcup V'_{\varepsilon }$ 
also consists of finitely many connected components.  
Let $W_{11}, \dots , W_{1l}$ be the connected components of 
${\hat{\F}}(W_{1})$ intersecting 
$V_{1, \varepsilon }$.  
Then 
$W := W_{11} \sqcup \dots \sqcup W_{1l} \subset V_{1, \varepsilon }$ 
is a neighbourhood of $\hat{L}_1 = \hat{\F}_H (x)$
with $W_{1i} \cap \hat{L}_1 \neq \emptyset$ 
for any $i = 1, \dots , l$.    
We show that $W$ is 
$\hat{\F}$-saturated. 
Indeed, 
since 
${\hat{\F}}(W) \cap V_{1, \varepsilon } = W$  
and 
$h( W_{1i}) \cap \hat{L}_1 \neq \emptyset$ 
for any $h \in H$ and $i = 1, \dots , l$, 
we have 
$V_{1, \varepsilon } \supseteq W = H(W)$. 
Since 
$\hat{\F}_H(W) \subseteq \overline{H(W)} \subseteq U_{1, \varepsilon}$, 
we have 
$\hat{\F}_H(W) \cap V'_{\varepsilon} = \emptyset$.  
Since $V_{1, \varepsilon} \sqcup V'_{\varepsilon}$ is 
an $\hat{\F}$-saturated neighbourhood of ${\hat{\F}}(W)$, 
we obtain 
$\hat{\F}_H(W) \subseteq V_{1, \varepsilon }$ 
and so 
$W = H(W) \subseteq \hat{\F}_H(W) \subseteq 
{\hat{\F}}(W) \cap V_{1, \varepsilon }  = W$. 
Thus 
$W$ is 
a desired $\hat{\F}_H$-saturated neighbourhood of $\hat{\F}_H(x)$ contained in 
$W \subseteq U_{1, \varepsilon } = B_{\varepsilon } (\hat{\F}_H (x))$. 
\end{proof}

\section{Codimension one or two results}

%
%
%
First, we consider the codimension one case.

\begin{proposition}\label{prop:003}
Let $G$ be an $R$-closed flow on a compact connected manifold $X$, 
$H$ a finite index normal subgroup of $G$, 
and 
$V$ the union of orbit closures of $H$ 
which are codimension one connected elements. 
If  
there is a nonempty connected component of $V$ which 
is open and consists of submanifolds,   
then $M/\hat{\mathcal{F}} _G$ is a closed interval or a circle 
such that 
there are at most two elements whose codimension more than one.  
\end{proposition}

Note that dimensions in the above statement are 
Lebesgue covering dimensions. 

\begin{proof}
By Lemma \ref{prop21}, 
we have that 
$H$ is also $R$-closed. 
Put $\hat{\mathcal{F}}  := \hat{\mathcal{F}} _H$. 
By Proposition \ref{prop:002}, 
we have that 
$M/\hat{\mathcal{F}} $ is Hausdorff. 
Let $U$ be the above open connected component of $V$. 
Put $p: M \to M/\hat{\mathcal{F}} $ the canonical projection. 
By Theorem 3.3 \cite{D2}, 
we have that 
$p(U)$ is a $1$-manifold. 
Since $\hat{\mathcal{F}} $ is $R$-closed, 
each connected component of 
the boundaries of $p(U)$ is a single point. 
Since $U$ is open, 
we obtain that 
the boundaries $\partial U := \overline{U} - \mathrm{int}U$ 
have at least codimension two 
and so that 
$M - \partial U$ is connected. 
%
This implies that 
$M/\hat{\mathcal{F}} $ is a closed interval or a circle. 
Since $G/H$ is a finite group acting $M/\hat{\mathcal{F}} $ 
and since a finite union of closures is a closure of finite union, 
we have that 
$M/\hat{\mathcal{F}} _G = (M/\hat{\mathcal{F}} _H)/(G/H)$ is either 
a closed interval or a circle and 
that 
there are at most two elements whose codimension more than one. 
\end{proof}

Second, we consider the codimension two case. 
%
%
 Consider the direct system $\{K_a\}$ 
 of compact subsets of a topological space $X$ and inclusion maps such that 
 the interiors of $K_a$ cover $X$.  
 There is a corresponding inverse system $\{ \pi_0( X - K_a ) \}$, 
 where $\pi_0(Y)$ denotes 
 the set of connected components of a space $Y$. 
 Then the set of ends of $X$ is defined to be the inverse limit of this inverse system.
By surfaces, we mean compact $2$-dimensional manifolds with conners
(i.e. locally modeled by $[0, 1]^2$).

\begin{proposition}
Let $G$ be an $R$-closed flow on a compact manifold $X$ and 
$H$ a finite index normal subgroup of $G$. 
Suppose that 
all orbit closures of $H$ are closed connected submanifolds. 
If 
all but finitely many closures have   
codimension two 
and finite exceptions have codimension more than two, 
then $M/\hat{\mathcal{F}} _G$ is a surface. 
\end{proposition}

\begin{proof} 
By Lemma \ref{prop21}, 
we have that 
$H$ is also $R$-closed. 
Put $\hat{\mathcal{F}}  := \hat{\mathcal{F}} _H$. 
Let $L_1, \dots , L_k$ be all higher codimension elements of $\hat{\mathcal{F}} $. 
Removing higher codimensional elements, 
let $M'$ be the resulting manifold, 
$\hat{\mathcal{F}} '$ the resulting decomposition of $M'$. 
Then $\hat{\mathcal{F}} '$ consists of codimension two closed connected submanifolds 
and is usc. 
By Theorem 3.12 \cite{D3},  
we have that 
$M'/\hat{\mathcal{F}} '$ is a surface $S'$.  
Then $(M/\hat{\mathcal{F}} ) - \{L_1, \dots , L_k \} \cong M'/\hat{\mathcal{F}} ' = S'$.  
We will show that 
$S'$ has $k$ ends. 
Indeed, 
since the exceptions $L_i$ are finite, 
there are pairwise disjoint neighbourhoods $U_i$ of them.
Since $\hat{\mathcal{F}} $ is usc, 
there are pairwise disjoint saturated neighbourhoods $V_i \subseteq U_i$ of them. 
Since $W_i - L_i$ is connected for any connected neighbourhoods $W_i$ of $L_i$, 
each end of $S'$ is corresponded to one of $L_i$. 
This shows that 
$S'$ has $k$ ends corresponding to $L_i$.  
Since $M/\hat{\mathcal{F}} $ is compact metrizable, 
we have that 
$M/\hat{\mathcal{F}} $ is an end compactification of $M'/\hat{\mathcal{F}} '$ and so 
a surface $S$. 
Since $G/H$ implies a finite group action on $S$ 
and since any finite union of closures is the closure of finite union, 
we obtain that 
$M/\hat{\mathcal{F}} _G \cong (M/\hat{\mathcal{F}} _H)/(G/H)$ is a surface. 
\end{proof}


\section{Toral minimal sets}

Recall that 
a vector field is trivial if it is identical or minimal. 
We obtain the following two statements for vector fields on $3$-manifolds. 
For a flow $v$  on a $3$-manifold, 
a minimal set $T$ of $v$ is called a torusoid for $v$ if 
there are 
an open annulus $A$ transverse to  $v$
and a circloid $C$ in $A$ whose saturation $\mathrm{Sat}_v(C)$ is $T$
with $C = T \cap A$.

\begin{lemma}\label{th:005}
Let $v$ be an $R$-closed vector field on a connected compact $3$-manifold $M$. 
Then the union of torusoids is open 
and the quotient space is $1$-dimensional. 
\end{lemma}

\begin{proof}
Let $A$ be an annular transverse manifold of $v$ 
and $C \subset A$ a circloid whose 
saturation is a torusoid $T$ with 
$C = T \cap A$. 
Take small connected neighbourhoods 
$U_1, U_2 \subseteq \mathrm{Sat}_{v}(A)$ of $T$ 
such that  the time one map $f_v: U_1 \cap T \to U_2 \cap T$ is well defined. 
Since $v$ is $R$-closed and so $\hat{\F}_{v}$ is usc, 
there is an $\hat{\F}_{v}$-saturated open neighbourhood 
$V \subseteq U_1 \cap U_2$ of $T$. 
Since $T$ is $\hat{\F}_{v}$-saturated and connected, 
the connected component $W$ of $V$ is also 
an $\hat{\F}_{v}$-saturated open neighbourhood of $T$. 
%
Since $W \subseteq V$, 
we have that 
$W_A := W \cap A \subseteq U_1\cap U_2 \cap A$ and so 
$f_{W_A} := f_v|: W_A \to W_A$ is well-defined. 
Since $v$ is $R$-closed, 
the mapping $f_{W_A}$ is a homeomorphism. 
Since $T \cap A = C \subset W_A$ is $f_v$ invariant 
and $U_1, U_2$ are small,  
we have that 
$W_A$ is connected. 
Let $B$ be the union of the connected components of $A - W_A$ 
which are contractible in $A$. 
Then $B \cup W_A$ is an open annulus 
and 
$\mathrm{Sat}_v(B \cup W_A)$ is an 
$\hat{\F}_{v}$-saturated connected open neighbourhood of $T$ 
which consists of torusoids. 
%
%
Indeed, 
define $\mathrm{Fill}_A(W_A)$ as follows:  
$ p \in \mathrm{Fill}_A(W_A)$ if 
there is a simple closed curve in $W \cap A$ which 
bounds  a disk in $A$ containing 
$p \in A$. 
Since each point of $B$ is bounded 
by a simple closed curve in $W_A$, 
we have that 
$B \sqcup W_A = \mathrm{Fill}_A(W_A)$ 
and 
$f' := f_v|: \mathrm{Fill}_A(W_A) \to \mathrm{Fill}_A(W_A)$ is a homeomorphism. 
Since $C$ is a circloid and $A$ is the annular neighbourhood, 
we obtain that 
$\mathrm{Fill}_A(W_A)$ is an open annulus. 
By the two point compactification of $\mathrm{Fill}_A(W_A)$,
we obtain the resulting sphere $S$ and 
the resulting homeomorphism $f_{S}$ with the two fixed points which are the added new points. 
Since $v$ is $R$-closed, 
we have that 
$M/\hat{v}$ is Hausdorff 
and so is $\mathrm{Fill}_A(W_A)/\hat{\F}_{f'}$. 
By the construction, $S/\hat{\F}_{f_S}$ is the 
two point compactification of $\mathrm{Fill}_A(W_A)/\hat{\F}_{f'}$. 
Since 
$M$ is normal, 
we obtain that 
$S/\hat{\F}_{f_S}$ is Hausdorff 
and 
so 
$f_S$ is $R$-closed. 
By Corollary 2.6 \cite{Y2}, 
we obtain that 
$S$ consists of two fixed points and 
circloids. 
Then the open neighbourhood 
$\mathrm{Sat}_{v}(B \sqcup W_A) 
= \mathrm{Sat}_{v}(\mathrm{Fill}_A(W_A))$ of $T$ 
consists of torusoids. 
Therefore 
the union of torusoids are open. 
\end{proof}

Recall that 
a minimal set is trivial if 
a single orbit or the whole manifold. 

\begin{lemma}\label{lem52} 
Let $v$ be an $R$-closed vector field on a connected compact $3$-manifold $M$. 
Suppose that each two dimensional minimal set is torusoid. 
If there is a torusoid,  
then the orbit class space $M/\hat{v}$ of $M$ is a closed interval or a circle. 
\end{lemma}

\begin{proof} 
Note that 
each minimal set which is not a torusoid is a closed orbit. 
By Lemma \ref{th:005}, 
the union of torusoids is open and 
the quotient space of it is 
a $1$-manifold. 
Since $v$ is $R$-closed, 
the each boundary component of it is a single minimal set. 
Since the each boundary component is a closed orbit and so is 
codimension more than $1$, 
the union of torusoids is connected. 
This implies that 
$M/\hat{v} $ is a closed interval or a circle. 
\end{proof}


\begin{lemma}\label{lem:006} 
Let $v$ be an $R$-closed vector field on a connected compact $3$-manifold $M$. 
Suppose that each two dimensional minimal set is torusoid. 
If there are at least three distinct periodic orbits, 
then $\mathrm{Per}(v)$ is open dense 
and 
$M = \mathrm{Sing}(v) \sqcup \mathrm{Per}(v)$. 
\end{lemma}

\begin{proof} 
Let 
$p:M \to M/\hat{v} $ be  the canonical projection.  
By Lemma \ref{lem52}, 
there are no two dimensional orbits closures. 
Therefore 
$M = \mathrm{Sing}(v) \sqcup \mathrm{Per}(v)$. 
Since $\mathrm{Sing}(v)$ is closed, 
we have that $\mathrm{Per}(v)$ is open. 
On the other hand, 
by Theorem 3.12 \cite{D3}, 
the quotient space of $\mathrm{Per}(v)$ is 
a $2$-dimensional manifold with conner. 
Since $M/\hat{v} $ is compact metrizable, 
by Urysohn's theorem, 
the Lebesgue covering dimension and 
the small inductive dimension are corresponded in $M/\hat{v} $.  
Hence 
the boundary $\partial(p(\mathrm{Per}(v)))$ of $p(\mathrm{Per}(v))$ has 
at most 
the small inductive dimension one. 
Since $\partial(p(\mathrm{Per}(v)))$ consists of 
singularities, 
we obtain that 
$\partial(\mathrm{Per}(v))$ has 
at most 
the small inductive dimension one and so 
the Lebesgue covering dimension one. 
Therefore 
$M - \partial(\mathrm{Per}(v))$ is connected 
and so $M - \partial(\mathrm{Per}(v)) = \mathrm{Per}(v)$. 
This shows that $\mathrm{Per}(v)$ is dense. 
\end{proof}

Now, we state the following trichotomy  
that $v$ is 
either 
``almost one dimensional'' 
or 
``almost two dimensional''
or 
with ``complicated'' minimal sets. 

\begin{proposition}
Let $v$ be an nontrivial $R$-closed vector field on a connected compact $3$-manifold $M$. 
Then one of the following holds: 
\\
1)
The orbit class space $M/\hat{v}$ of $M$ is a closed interval or a circle and 
each interior point of the orbit class is two dimensional. 
\\
2) 
$\mathrm{Per}(v)$ is open dense 
and 
$M = \mathrm{Sing}(v) \sqcup \mathrm{Per}(v)$. 
\\
3)
There is a nontrivial minimal set which is not a torusoid. 
\end{proposition}

\begin{proof}
Suppose that each nontrivial minimal set is a torusoid. 
Since $v$ is nontrivial, 
there is a minimal set. 
If the minimal set is two dimensional, 
then 
Lemma \ref{lem52} implies that 1). 
Otherwise 
we may assume that 
there are no two dimensional minimal sets. 
Since $v$ is nontrivial, 
there is a periodic orbit. 
By flow box theorem, 
this has a neighbourhood without singularities. 
Thus there are infinitely many periodic orbits. 
By Lemma \ref{lem:006}, 
we have that 2) holds.
\end{proof}

%
%


\end{document}